\documentclass[11pt]{article}
\usepackage{graphicx}
\usepackage{verbatim}
\usepackage{fullpage}
\usepackage{hyperref}
\usepackage{amssymb}
\usepackage{amsmath}
\usepackage{amsthm}

\def\final{0}  
\def\iflong{\iffalse}
\ifnum\final=0  
\newcommand{\knote}[1]{[{\tiny Karthik: \bf #1}]\marginpar{*}}
\newcommand{\sidecomment}[1]{}
\else 
\newcommand{\knote}[1]{}
\newcommand{\sidecomment}[1]{}
\fi  

\newtheorem{theorem}{Theorem}

\newtheorem{lemma}[theorem]{Lemma}

\newtheorem{definition}[theorem]{Definition}

\def\reals{\mathbb{R}}

\def\Z{\mathbb{Z}}
\def\E{\mathbb{E}}

\def\H{\mathcal{H}}
\def\C{\mathbb{C}}

\def\det{\mathrm{det}}

\def\Ht{\{0,1,2\}}
\def\sgn{\mathrm{sgn}}
\def\sym{\mathrm{sym}}
\def\poly{\mathrm{poly}}

\title{Towards Constructing Ramanujan Graphs Using Shift Lifts}
\author{Karthekeyan Chandrasekaran \thanks{Department of Industrial and Enterprise Systems Engineering, University of Illinois, Urbana-Champaign. {\tt karthe@illinois.edu}. Part of this work was done while the author was a Simons postdoctoral fellow at Harvard University. }  \and Ameya Velingker\thanks{Computer Science Department, Carnegie Mellon University. {\tt avelingk@cs.cmu.edu}. Part of this work was done while the author was a visiting research fellow at Harvard University. Research supported in part by NSF grant CCF-0963975.}}
\date{}

\begin{document}
\maketitle
\begin{abstract}
In a breakthrough work, Marcus et al. \cite{MSS13} recently showed that every $d$-regular bipartite Ramanujan graph has a 2-lift that is also $d$-regular bipartite Ramanujan. As a consequence, a straightforward iterative brute-force search algorithm leads to the construction of a $d$-regular bipartite Ramanujan graph on $N$ vertices in time $2^{O(dN)}$. 
Shift $k$-lifts studied in \cite{AKM} lead to a natural approach for constructing Ramanujan graphs more efficiently. The number of possible shift $k$-lifts of a $d$-regular $n$-vertex graph is $k^{nd/2}$. Suppose the following holds for $k=2^{\Omega(n)}$:
\begin{align}
\mbox{There} & \mbox{ exists a shift $k$-lift that maintains the Ramanujan property of}\notag\\ 
 & \mbox{$d$-regular bipartite graphs on $n$ vertices for all $n$.}\tag{$\star$}\label{conj:shift-lift}
\end{align}
Then, by performing a similar brute-force algorithm, one would be able to construct an $N$-vertex bipartite Ramanujan graph in time $2^{O(d\, {\log^2 N})}$. 
Also, if (\ref{conj:shift-lift}) holds for all $k\geq 2$, then one would obtain an algorithm that runs in $\poly_d(N)$ time.
In this work, we take a first step towards proving (\ref{conj:shift-lift}) by showing the existence of shift $k$-lifts that preserve the Ramanujan property in $d$-regular bipartite graphs for $k=3$, $4$. 



\end{abstract}

\section{Introduction}

Expander graphs have generated much interest during the last several decades in many areas such as network design, cryptography, complexity theory and coding theory. The ability to efficiently construct expander graphs has widespread applications (\cite{Spi96}, \cite{AKS83}, \cite{ABN92}, \cite{Has99}, \cite{Guru09}, \cite{crypto}, \cite{Zem94}, \cite{Din07}). 
Sparse expander graphs are significant from the perspective of these applications. In particular, the expansion properties of regular graphs have been well-studied. 
For $d$-regular graphs $G$, the largest eigenvalue of the adjacency matrix $A_G$ is $d$ and it is referred to as the \emph{trivial} eigenvalue.  
A large difference between $d$ and the largest (in absolute value) non-trivial eigenvalue $\lambda$ of $A_G$ implies better expansion. 
By the Alon-Boppanna bound \cite{Nil91}, we have that $\lambda \ge 2\sqrt{d-1}-o(1)$ as the graph size increases. 
Consequently, graphs with $\lambda\le 2\sqrt{d-1}$ are optimal expanders. Such graphs are known as \emph{Ramanujan} graphs.

Until recently, Ramanujan graphs were known to exist only for very restricted values of the degree $d$ \cite{LPS88}, \cite{Mar88}, \cite{Piz90}, \cite{Chi92}, \cite{JL97}, \cite{Mor94}. In a breakthrough work, Marcus et al. \cite{MSS13} showed the existence of an infinite family of $d$-regular bipartite Ramanujan graphs for every $d\ge 2$. The work used a graph operation known as a \emph{2-lift} that was introduced by Bilu and Linial \cite{BL06}. The 2-lift operation doubles the size of the graph while preserving the $d$-regular and bipartiteness properties. Marcus et al. showed that there exists a 2-lift that also preserves the Ramanujan property of every Ramanujan base graph. Thus, their result shows the existence of an infinite family of $d$-regular bipartite Ramanujan graphs containing $N$ vertices, where $N=2^{i}n$ for every $i=0,1,2,\ldots$, and $n$ is the number of vertices in an arbitrary $d$-regular bipartite Ramanujan graph. In particular, the complete bipartite graph on $2d$ vertices, $K_{d,d}$, is a $d$-regular bipartite Ramanujan graph that can be used as the base graph with $n$ vertices. 
The best currently-known algorithm for constructing an $N$-vertex Ramanujan graph for \emph{arbitrary} degree $d$ following the existential proof of Marcus et al. \cite{MSS13} is by a brute-force search: Start with $K_{d,d}$ as the base graph and iteratively find a $2$-lift of the current graph such that all the new eigenvalues of the lifted graph are at most $2\sqrt{d-1}$; in order to obtain an $N$-vertex graph, we need to perform $\log_2 (N/2d) $ iterations. The running time of the algorithm is dominated by the brute-force search over $2^{O(dN)}$ possible 2-lifts in the final iterative step (along with a $\poly(N)$-time check for each possible 2-lift to verify whether the eigenvalues are all $\leq 2\sqrt{d-1}$).

The family of \emph{shift $k$-lifts} gives a natural approach for a faster construction of an $N$-vertex $d$-regular bipartite Ramanujan graph. An extension of the 2-lift operation, the shift $k$-lift is a graph operation that increases the number of vertices of a graph by a factor of $k$ while preserving the $d$-regular and  bipartiteness properties of the base graph.  The existence of a shift $k$-lift that maintains the Ramanujan property of $d$-regular bipartite graphs on $n$ vertices for $k=2^{\Omega(n)}$ is sufficient to construct, for all $N$, an $N$-vertex bipartite Ramanujan graph in time $2^{O(d \log^2 N)}$ by a straightforward brute-force search: Again, start with a $(\log N)$-vertex Ramanujan graph as the base graph and do one shift lift for $k = N/\log N$. The $(\log N)$-vertex Ramanujan base graph can be obtained by starting from $K_{d,d}$ and repeating the same construction recursively (or by repeatedly finding admissible 2-lifts).

Furthermore, note that if it is possible to prove the existence of a shift $k$-lift of any $d$-regular bipartite Ramanujan graph that preserves the Ramanujan property for \emph{every} $k\geq 2$, then we would be able to obtain $\Omega(N)$-vertex bipartite Ramanujan graphs in time $(N/d)^{O(d^2)}$, i.e., \emph{polynomial} in $N$, by performing a single brute-force search to find an admissible shift $k$-lift of $K_{d,d}$ for $k = O(N/d)$.

In this work, we take a first step towards this existential result by showing it for $k=3$ and $k=4$. We note that the case of $k=3$ was also shown simultaneously in \cite{LPV14}.

The structure of this work is as follows: Section~\ref{sec:definitions} provides an overview of the main theorem along with the definitions as well as the main ideas that will be used in the proof. Section \ref{sec:prelims} gives the definitions, notation and some of the results that will be used in the proof of the main result. Section~\ref{sec:shift3lift} shows that the expected characteristic polynomial for uniformly random shift 3-lifts is the \emph{matching polynomial} of the base graph, whose roots are well-behaved.
Section~\ref{sec:shift4lift} proves that the expected characteristic polynomial for a subset of uniformly random shift $4$-lifts is once again the \emph{matching polynomial} of the base graph. 
Subsequently, Section~\ref{sec:interlacing} uses the results of the previous sections along with the \emph{method of interlacing polynomials} to prove the main theorem. Finally, Section~\ref{sec:discussion} discusses the main result and possible extensions to consider.



\subsection{Main Results}\label{sec:definitions}
The maximum and minimum eigenvalues of the adjacency matrix of a $d$-regular bipartite graph $G$ are $d$ and $-d$. These eigenvalues are known as the trivial eigenvalues of $G$. A $d$-regular graph $G$ is known to be \emph{Ramanujan} if every non-trivial eigenvalue of the adjacency matrix of $G$ lies between $-2\sqrt{d-1}$ and $2\sqrt{d-1}$. A $k$-lift of a graph $G=(V,E)$ is a graph $H$ obtained as follows: for each vertex $v\in V$, create $k$ copies of $v^1,\ldots, v^k$ in $H$; orient the edges of $G$ arbitrarily and for each edge $(u,v) \in E$, pick a permutation $\sigma_{uv} \in S_k$ and add edges $u^iv^{\sigma_{uv}(i)}$ to $H$. We consider a strict subset of $k$-lifts known as \emph{shift $k$-lifts} as defined by Agarwal et al. \cite{AKM}. Shift $k$-lifts are those lifts for which the permutation for every edge corresponds to a shift permutation, namely, there exists a \emph{shift function} $s: E\to [k]$ such that $\sigma_{uv}(i) = (i+s(u,v)) \bmod k$ for all $(u,v)\in E$. We will often refer to a shift $k$-lift by its shift function $s:E \to [k]$.

Our main result is stated below.

\begin{theorem}\label{3-4-lifts-preserve-Ramanujan}
If $G=(V,E)$ is a $d$-regular bipartite Ramanujan graph, then there exist a shift $3$-lift and a shift $4$-lift of $G$, both of which are also $d$-regular, bipartite, and Ramanujan.
\end{theorem}

\noindent {\bf Remark 1.}
We note that our result for shift $4$-lift is not an immediate consequence of Marcus et al.'s result showing the existence of a $2$-lift that preserves the Ramanujan property. The set of shift $4$-lifts cannot be obtained by considering $2$-lifts of $2$-lifts. The number of possible $4$-lifts of a graph containing a single edge is $4$, while the number of possible $2$-lifts of $2$-lifts of the same graph is $8$. Also, shift $4$-lifts are not a strict subset of $2$-lifts of $2$-lifts (see figures \ref{fig:shift-4-lifts} and \ref{fig:two-2-lifts}). 
\begin{figure}[ht]
\centering
\includegraphics[scale=0.5]{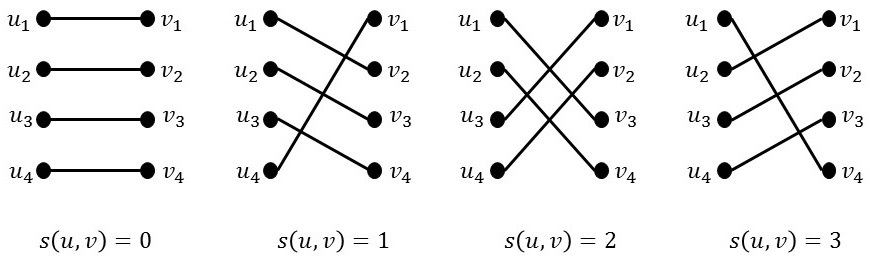}
\caption{The four possibilities that can be obtained by considering shift $4$-lifts of a single edge}\label{fig:shift-4-lifts}
\end{figure}

\begin{figure}[ht]
\centering
\includegraphics[scale=0.5]{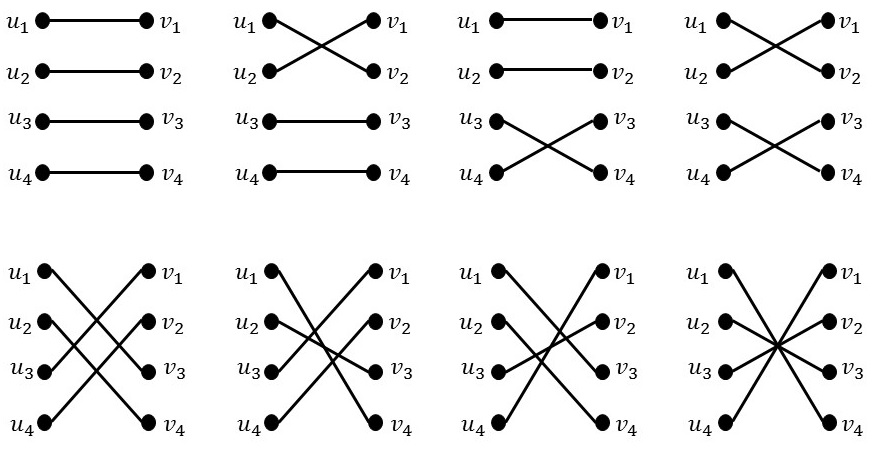}
\caption{The eight possibilities that can be obtained by considering $2$-lifts of $2$-lifts of a single edge}\label{fig:two-2-lifts}
\end{figure}

\noindent {\bf Remark 2.}
For any fixed shift $k$-lift, the spectrum of the lifted graph is the union of the eigenvalues of the adjacency matrix of $G$ and $k-1$ other matrices (see Theorem \ref{eigenvalues-of-shifts}). 
Using the technique of Marcus et al., it is easy to show that for each of these matrices, there is a shift that ensures that the eigenvalues of that matrix are $\leq 2\sqrt{d-1}$. 
The main contribution of this work is showing the existence of a shift that bounds the eigenvalues of several matrices \emph{simultaneously}. \\

It is straightforward to verify that $k$-lifts preserve the $d$-regular and bipartiteness properties. Our main goal is to show the existence of a shift $3$-lift and a shift $4$-lift that preserves the Ramanujan property. Similar to Marcus et al. \cite{MSS13}, we note that since the lifts are also bipartite, the eigenvalues of the corresponding adjacency matrix also occur in pairs $(\lambda,-\lambda)$. Consequently, it suffices to show that there exists a lift whose largest non-trivial eigenvalue is at most $2\sqrt{d-1}$. 

\subsection{Preliminaries}\label{sec:prelims}
Let $G=(V,E)$ be the base graph, with $|V|=n$. We identify the vertices with the elements of $\{1, 2, \dots, n\}$. Let us fix an arbitrary ordering of the vertices in the graph and orient the edges $(u,v)$ such that $u<v$. For notational convenience, let $|E|=m$, and let the edges be $e_1,\ldots, e_m$.
Recall that a shift $k$-lift of $G$ can be described by a shift function $s: E\to\{0,1,\dots,k-1\}$. Note that a shift function $s$ satisfies $s(u,v) \equiv -s(v,u) \pmod k$ for any oriented edge $(u,v)\in E$.

The eigenvalues of the adjacency matrix of a shift $k$-lift of $G$ can be characterized by the eigenvalues of $k$ matrices \cite{AKM},\cite{MS95}.
Let $A$ denote the adjacency matrix of $G$, and let $s$ be the shift function for a shift $k$-lift. Now for a variable $t$, define the following $n \times n$ matrix $A_s(t)$:
\begin{eqnarray*}
A_s(t)[i,j]=
\begin{cases}
0 &\text{ if } A_{ij}=0,\\
t^{s(i,j)} &\text{ if } A_{ij}=1.
\end{cases}
\end{eqnarray*}

\begin{theorem}\cite{AKM},\cite{MS95}\label{eigenvalues-of-shifts}
Let $H$ denote the shift $k$-lift obtained from $G$ using the shift function $s$. Then, the eigenvalues of the adjacency matrix of $H$ are the union of the eigenvalues of the matrices $(A_s(\omega^i))_{i=0,1,\ldots,k-1}$, where $\omega$ corresponds to a primitive $k$-th root of unity.
\end{theorem}

We define a \emph{matching} of $G$ to be a set $M = \{(u_1, v_1), (u_2, v_2), \dots, (u_r, v_r)\}$ of edges such that no vertex is adjacent to more than one edge in $M$. A \emph{perfect matching} is a matching in which each vertex is adjacent to exactly one edge.

\begin{definition}
 We define the \emph{matching polynomial} of a graph $G$ on $n$ vertices to be
 \[
  \mu_G(x) := \sum_{k=0}^{\lfloor n/2\rfloor} (-1)^k m_k x^{n-2k},
 \]
where $m_k$ is the number of matchings of $G$ with exactly $k$ edges.
\end{definition}

We need the following result bounding the maximum root of the matching polynomial. 

\begin{theorem}\cite{HL72}\label{matching-polynomial-roots}
If $G$ is a $d$-regular graph, then the maximum root of the matching polynomial, $\mu_G(x)$, is at most $2\sqrt{d-1}$.
\end{theorem}

The family of $2$-lifts is identical to the family of shift $2$-lifts. We need the following result showing the existence of $2$-lifts whose new eigenvalues satisfy the Ramanujan bound.
\begin{theorem} \cite{MSS13} \label{thm:2-lift}
If $G$ is a $d$-regular bipartite graph, then there exists a $2$-lift of $G$ such that all new eigenvalues of the adjacency matrix of the lift are at most $2\sqrt{d-1}$.
\end{theorem}

We also need the following result showing the real-rootedness of the expected characteristic polynomial of a matrix obtained as the sum of rank one matrices, where the rank one matrices are chosen according to a product distribution (we use $a^*$ to denote the conjugate transpose).
\begin{theorem}[Corollary 4.3 in \cite{MSS13-2}]\label{stronger-real-rootedness}
Let $\H$ be a finite set, $a_j^r\in \C^n$, $p_j^r\in [0,1]$ for $j\in [m], r\in \H$ such that $\sum_{r\in \H}p_j^r=1\ \forall j\in [m]$. Then every polynomial of the form
\[
P(x):= \sum_{t_1,\ldots, t_m\in \H}\left(\prod_{r\in \H}\prod_{j\in[m]:t_j=r}p_j^r\right)\det\left[xI+\sum_{r\in\H}\sum_{j\in[m]:t_j=r}a_j^r({a_j^r})^*\right]
\]
is real rooted.
\end{theorem}
We need the notion of \emph{common interlacing}. We say that a polynomial $g(x)=\prod_{i=1}^{n-1}(x-\alpha_i)$ interlaces a polynomial $f(x)=\prod_{i=1}^{n}(x-\beta_i)$ if $\beta_i\le \alpha_i\le\beta_{i+1}$ for every $i=1,\ldots, n-1$.We say that a family of univariate polynomials $f_1(x), f_2(x),\ldots, f_r(x)$ have a common interlacing if there exists a polynomial $g(x)$ such that $g(x)$ interlaces $f_i(x)$ for every $i\in [r]$.

We define $S_k$ to be the \emph{symmetric group} on $k$, i.e., the group of permutations of $\{1,2,\dots, k\}$. Furthermore, for a finite set $T$, we will define $\sym(T)$ to be the set of permutations of $T$. Also, for a permutation $\pi$, we let $\sgn(\pi)$ denote the \emph{sign} of $\pi$ (+1 if $\pi$ is even, and -1 if $\pi$ is odd).

\section{Shift 3-Lifts and the Matching Polynomial} \label{sec:shift3lift}
By Theorem \ref{eigenvalues-of-shifts}, it is sufficient to show that there exists a shift function $s:E \to \{0,1,2\}$ such that the eigenvalues of $A_s(\omega)$ and $A_s(\omega^2)$ are at most $2\sqrt{d-1}$, where $\omega$ is a cube root of unity. We first observe that $A_s(\omega)=A_s(\omega^2)^T$. Since the determinant of a matrix is preserved under the transpose operation, the characteristic polynomials of both matrices are identical. Thus, it suffices to show that there exists a shift $s$ such that the largest eigenvalue of $A_s(\omega)$ is at most $2\sqrt{d-1}$. The rest of the proof technique is a natural extension of the one by Marcus et al. \cite{MSS13}. We work out the details for the sake of completeness. 

We first show that the expected characteristic polynomial of $A_s(\omega)$, where the shift values $s(u,v)$ for all edges $(u,v)\in E$ are chosen uniformly at random from $\{0,1,2\}$, is the matching polynomial for $G$.

\begin{lemma}\label{matching-polynomial-3-lifts}
Suppose for every edge $(u,v)$, let $s(u,v)$ be chosen uniformly at random from $\{0,1,2\}$. Then, $\E_s(\det[xI-A_s(\omega)])=\mu_G(x)$.
\end{lemma}
\begin{proof}
For notational convenience, let $B_s(x) = xI - A_s(\omega)$, and let $B_s(x)_{u,v}$ denote the $(u,v)$ entry of $B_s(x)$. Then, note that
\begin{align}
 \E_s(\det[xI-A_s(\omega)]) &= \E_s\left(\sum_{\pi\in S_n} {\sgn(\pi)} \cdot\prod_{j=1}^n B_s(x)_{j, \pi(j)}\right) \nonumber\\
 &= \sum_{\pi\in S_n} {\sgn(\pi)} \cdot \E_s\left(\prod_{j=1}^n B_s(x)_{j,\pi(j)}\right). \label{eq:expdet}
\end{align}
Let $(u,v)$ be an edge in $G$. Suppose $\pi$ is a permutation such that $\pi(u) = v$ but $\pi(v) \neq u$, then
\begin{eqnarray*}
 \E_s\left(\prod_{j=1}^n B_s(x)_{j,\pi(j)}\right) &=& \E_s(B_s(x)_{u,v}) \cdot \E_s\left(\prod_{j\neq u} B_s(x)_{j,\pi(j)}\right)\\
 &=& \E_{s(u,v)} (\omega^{s(u,v)}) \cdot \E_s\left(\prod_{j\neq u} B_s(x)_{j,\pi(j)}\right)\\
 &=& \frac{\omega^0 + \omega^1 + \omega^2}{3} \cdot \E_s\left(\prod_{j\neq u} B_s(x)_{j,\pi(j)}\right)\\
 &=& 0.
\end{eqnarray*}
Similarly, if $\pi$ is a permutation satisfying $\pi(v) = u$ but $\pi(u)\neq v$, then 
\[
 \E_s\left(\prod_{j=1}^n B_s(x)_{j,\pi(j)}\right) = 0.
\]

Thus, the only permutations for which the expectation on the right hands side of (\ref{eq:expdet}) is non-zero are those $\pi$ that correspond to \emph{matchings}: There exists a matching $M$ such that for every edge $(u,v)\in M$, we have $\pi(u)=v$ and $\pi(v)=u$ and for every vertex $w\in V$ which is not adjacent to any of the edges in $M$, we have $\pi(w)=w$. 
Moreover, if $\pi$ corresponds to a matching $M = \{(u_1, v_1), (u_2, v_2), \dots, (u_t, v_t)\}$, then
\begin{align*}
 \mathbb{E}_s\left(\prod_{j=1}^n B_s(x)_{j,\pi(j)}\right) &= x^{n-2t}\cdot \prod_{j=1}^t \E_s\left(B_s(x)_{u_j, v_j} \cdot B_s(x)_{v_j, u_j}\right) \\
 &= x^{n-2t} \cdot \prod_{j=1}^t \E_s \left(\omega^{s(u_j,v_j)}\cdot \omega^{-s(u_j, v_j)}\right)\\
 &= x^{n-2|M|}.
\end{align*}
Thus, from (\ref{eq:expdet}), we conclude that
\begin{align*}
 \E_s(\det[xI-A_s(\omega)]) &= \sum_{\text{$\pi\in S_n$ corresponds to a matching $M$}} {\sgn(\pi)}\cdot x^{n-2|M|}\\
 &= \sum_{\text{$\pi\in S_n$ corresponds to a matching $M$}} (-1)^{|M|} x^{n-2|M|}\\
 &= \sum_{k=0}^{\lfloor n/2\rfloor} (-1)^k m_k x^{n-2k}\\
 &= \mu_G(x),
\end{align*}
where $m_k$ is the number matchings in $G$ with exactly $k$-edges.
\end{proof}

\section{Shift 4-Lifts and the Matching Polynomial} \label{sec:shift4lift}
For the case of shift $4$-lifts, it is sufficient to show that there exists a shift function $s: E \to \{0,1,2,3\}$ such that the eigenvalues of $A_s(i), A_s(-1), A_s(-i)$ are at most $2\sqrt{d-1}$, where $i$ is the complex square root of $-1$. Once again, we note that $A_s(i)=A_s(-i)^T$ for any fixed shift function $s$.  Therefore, it suffices to show the existence of a shift function $s: E\to \{0,1,2,3\}$ for which the eigenvalues of $A_s(-1)$ and $A_s(i)$ are at most $2\sqrt{d-1}$. 

We show the existence of a shift that satisfies the eigenvalue bound for both matrices simultaneously by a two-step procedure: Using the result of Marcus et al., we have a shift function $b: E\to\{0,1\}$ corresponding to a shift 2-lift such that the eigenvalues of $A_b(-1)$ are at most $2\sqrt{d-1}$. We then show that there exists an $s:E\to\{0,2\}$ such that for $s' = s+b$, the eigenvalues of $A_{s'}(i)$ are at most $2\sqrt{d-1}$. It is straightforward to verify that $A_{s'}(-1)=A_b(-1)$ since for any edge $(u,v)\in E$, we have $(-1)^{s'(u,v)} = (-1)^{s(u,v)+b(u,v)} = (-1)^{b(u,v)}$, using the fact that $s(u,v) \in \{0,2\}$. We thus have a shift 4-lift given by the shift function $s':E\to\{0,1,2,3\}$ such that the eigenvalues of $A_{s'}(-1)$ and $A_{s'}(i)$ are at most $2\sqrt{d-1}$, thereby giving us the desired conclusion.

We now proceed to show that for any fixed $b:E\to\{0,1\}$, the expected characteristic polynomial of $A_{s'+b}(i)$, where $s$ is chosen uniformly at random over the function space $\{E\to\{0,2\}\}$, is the matching polynomial.
\begin{lemma}\label{matching-polynomial-4-lifts}
Given $b:E\to\{0,1\}$, let $s:E\to\{0,2\}$ be chosen uniformly at random from $\{E\to\{0,2\}\}$, and set $s'=s+b$. Then,
\[
\E_{s}(\det[xI-A_{s'}(i)])=\mu_G(x).
\]
\end{lemma}

\begin{proof}
We use the Leibniz expansion of the characteristic polynomial:
\[
\E_{s}(\det[xI-A_{s'}(i)])
\]
\begin{align}
&= \E_{s}\left(\sum_{\sigma \in S_n}{\sgn(\sigma)} \cdot\prod_{u=1}^n{(xI-A_{s'}(i))_{u,\sigma(u)}}\right) \nonumber\\
&= \sum_{k=0}^n x^{n-k} \sum_{T\subseteq [n], |T|=k} \sum_{\substack{\sigma\in \sym(T)\\ \sigma(u)\neq u\ \forall u\in T}} {\sgn(\sigma)}\cdot \E_{s}\left(\prod_{u\in T} A_{s'}(i)_{u,\sigma(u)}\right). \label{eq:charpolyexp}
\end{align}
Now we observe that 
\begin{eqnarray*}
\E_{s}(A_{s'}(i)_{u,v})=\E_{s}(A_{s+b}(i)_{u,v})=\begin{cases}
0 & \text{ if } (u,v)\not\in E,\\
\E_{s(u,v)}(i^{s(u,v)+b(u,v)}) & \text{ if }(u,v)\in E.
\end{cases}
\end{eqnarray*}
Moreover, for $(u,v)\in E$,
\[
\E_{s(u,v)}(i^{s(u,v)+b(u,v)}) = i^{b(u,v)} \E_{s(u,v)\in \{0,2\}}(i^{s(u,v)}) = i^{b(u,v)} \cdot 0 = 0.
\]
Since the values $s({u,v})$ for different edges $(u,v)$ are independent, we have that 
\[
\E_{s}\left(\prod_{u\in T} A_{s'}(i)_{u,\sigma(u)}\right) = 0
\]
if $\sigma(\sigma(u))\neq u$ for any $u\in T$. Therefore, the only non-zero terms in the sum (\ref{eq:charpolyexp}) are the ones corresponding to permutations $\sigma\in\sym(T)$ such that $\sigma(\sigma(u))=u$ for all $u\in T$. We note that such a $\sigma$ corresponds to a perfect matching $M = \{(u_{i_1}, v_{i_1}), (u_{i_2}, v_{i_2}), \dots, (u_{j_{|T|/2}}, v_{j_{|T|/2}})\}$ over the vertices of $T=\{u_{j_1},\ldots u_{j_{|T|/2}}, v_{j_1},\ldots v_{j_{|T|/2}}\}$ using the edges of $G$, i.e., $\sigma(u_{j_k}) = v_{j_k}$ and $\sigma(v_{j_k}) = u_{j_k}$ for every $k\in \{1,2,\ldots,|T|/2\}$. In this case,
\begin{align*}
\E_{s}\left(\prod_{u\in T} A_{s'}(i)_{u,\sigma(u)}\right)
&= \prod_{k=1}^{|T|/2} \E_{s}\left(A_{s'}(i)_{u_{j_k},\sigma(u_{j_k})}\cdot A_{s'}(i)_{v_{j_k},\sigma(v_{j_k})}\right)\\
&= \prod_{k=1}^{|T|/2} \E_{s}\left(A_{s+b}(i)_{u_{j_k}, v_{j_k}}\cdot A_{s+b}(i)_{v_{j_k}, u_{j_k}}\right)\\
&= \prod_{j=1}^{|T|/2} \E_{s}\left(i^{s(u_{j_k},v_{j_k})+b(u_{j_k},v_{j_k})} \cdot i^{s(v_{j_k},u_{j_k})+b(v_{j_k},u_{j_k})}\right)\\
&= \prod_{j=1}^{|T|/2} \E_{s}(1)
=1.
\end{align*}
The penultimate equality is because $s(v_j, u_j) = -s(u_j, v_j)$ and $b(v_j,u_j) = -b(u_j,v_j)$.

Consequently, we have that
\begin{align*}
\E_{s}(\det[xI-A_{s'}(i)]) 
&= \sum_{k=0}^n x^{n-k} \sum_{T\subseteq [n], |T|=k} \sum_{\substack{\sigma\in \sym(T):\\ \sigma(\sigma(u))=u, \sigma(u)\neq u\ \forall u\in T}} {\sgn(\sigma)}\\
&= \sum_{k=0}^n x^{n-k} \sum_{T\subseteq [n], |T|=k} \sum_{\substack{M:M \text{ is a perfect}\\ \text{matching over the vertices }T}} (-1)^{|T|/2}\\
&= \sum_{k=0}^{\lfloor n/2\rfloor} (-1)^k m_k x^{n-2k}
=\mu_G(x),
\end{align*}
where $m_k$ is the number matchings in $G$ with exactly $k$-edges.
\end{proof}

\section{Interlacing Families} \label{sec:interlacing}
We proceed in a fashion similar to the proof technique of Marcus et al. Let $b: E \to \{0,1\}$ be the shift function for some $2$-lift. For a shift function $s: E\to\{0,1\}$, we will use the shorthand notation $s_j = s(e_j)$ for $1\leq j \leq m$. For a fixed $b$, shift function $s$, and $t\in\C$, we define
\begin{align*}
f_{s_1,\ldots,s_m}^{(b,t)}(x)&:=\det[xI-A_{b+s}(t)]
\end{align*}
We will also need to fix a set $\mathcal{H}$ from which we can choose shift values. 
For the case of shift $3$-lifts, we will consider $b$ as the shift function that maps each edge to zero, and we will also set $\mathcal{H}=\{0,1,2\}$ and $t=\omega = e^{2\pi i/3}$. For the case of shift $4$-lifts, we will need to consider arbitrary $b$, and we will also set $\mathcal{H}=\{0,2\}$ and $t=i = \sqrt{-1}$.

Now, we define a family of polynomials: For every $k\in \{1,\ldots,m-1\}$ and every fixed partial assignment $s_1,\ldots, s_k \in \H$ for $s$, let 
\begin{align*}
f_{s_1,\ldots,s_k}^{(b,t)}(x)&:=\sum_{s_{k+1},\ldots,s_m\in \H}f_{s_1,\ldots, s_m}^{(b,t)}(x),
\end{align*}
i.e., the sum of $f_{s_1,\dots,s_m}^{(b)}(x)$ over all possible shift functions $s$ taking values in $\mathcal{H}$ that agree with the partial assignment for the first $k$ edges.
Also let 
\begin{align*}
f_{\emptyset}^{(b,t)}(x)&:=\sum_{s_{1},\ldots,s_m\in \H}f_{s_1,\ldots, s_m}^{(b,t)}(x). 
\end{align*}
Also, for ease of notation, we will omit $b$ whenever $b$ is understood to be the constant function that is zero on all edges, i.e., $f_{s_1,\dots, s_k}^{(t)} := f_{s_1, \dots, s_k}^{(b,t)}$ for any $k=1,\dots, m$ and $f_{\emptyset}^{(t)} = f_{\emptyset}^{(b,t)}$, where $b:E\to\{0,1\}$ is the function defined by $b(e) = 0$ for all $e\in E$.

We need the following theorem that shows the existence of a favorable path in the family of polynomials defined above. 
\begin{theorem}[Theorem 4.4 in \cite{MSS13}]\label{common-interlacing-implies-good-signing}
Fix $b:E\to\{0,1\}$ and $t\in\C$. Suppose that for every $k=0,1,\ldots,m-1$ and every $s_1,\ldots, s_k\in \H$, the polynomials
\[
\left(f_{s_1,\ldots,s_k,s_{k+1}=r}^{(b,t)}(x)\right)_{r\in \H}
\]
have positive leading coefficients, are real-rooted and have a \emph{common interlacing}. Then, there exists $s_1,\ldots, s_m\in \H$ such that the largest root of $f_{s_1,\ldots, s_m}^{(b,t)}(x)$ is at most the largest root of $f_{\emptyset}^{(b,t)}(x)$.
\end{theorem}

We need the following result to show the existence of a common interlacing. 
\begin{lemma}[Corollary 1.36 in \cite{Fisk}]\label{conv-comb-implies-interlacing}
Let $f_1,\ldots, f_t$ be univariate polynomials of degree $n$ such that, for all $\alpha_1,\ldots, \alpha_t$ that are non-negative, the sum $\sum_{r=1}^t\alpha_r f_r$ has all real roots. Then $(f_r)_{r\in \{1,\ldots,t\}}$ have a common interlacing.
\end{lemma}
Multiplication by a non-zero constant does not change the roots of a polynomial. Hence, an equivalent condition to the one in the hypothesis of Lemma~\ref{conv-comb-implies-interlacing} is that for every non-negative $p_1,\ldots, p_t$ such that $\sum_{r=1}^t p_r=1$, the sum polynomial $\sum_{r\in \H} p_r f_r $ has all real roots.

The following lemmas prove the hypothesis of Theorem \ref{common-interlacing-implies-good-signing} for the cases of shift $3$-lifts and shift $4$-lifts.
\begin{lemma}\label{real-rootedness-for-3}
Fix $t=\omega = e^{2\pi i/3}$.
For every $k=0,1,\ldots, m-1$, every $s_1,\ldots, s_k\in \{0,1,2\}$, and every non-negative $\alpha_0,\alpha_1,\alpha_2$ such that $\alpha_0+\alpha_1+\alpha_2=1$, the polynomial
\[
\sum_{s_{k+1}\in \Ht} \alpha_{s_{k+1}} f_{s_1,\ldots,s_k,s_{k+1}}^{(t)}(x)
\]
has all real roots.
\end{lemma}
\begin{proof}
We will use Theorem \ref{stronger-real-rootedness} with appropriately chosen $a_j^r$ and $p_j^r$, for $j\in E$ and $r\in \H=\{0,1,2\}$. Fix $k\in \{0,1,\ldots,m-1\}$, partial assignment $s_1,\ldots, s_k\in\{0,1,2\}$ and values $\alpha_0,\alpha_1,\alpha_2\ge 0$ such that $\alpha_0+\alpha_1+\alpha_2=1$. Now, for edge $j=(u,v)$, we take $a_j^0=e_u-e_v$, $a_j^1=e_u-\omega^2e_v$, $a_j^2=e_u-\omega e_v$, where $e_u\in \reals^n$ is the indicator vector of vertex $u$. Moreover, we will take 
\begin{eqnarray*}
p_j^r=
\begin{cases}
1 &\text{ if } j\le k, s_j=r\\
0 &\text{ if } j\le k, s_j\in \{0,1,2\}\setminus \{r\}\\
\alpha_r &\text{ if } j= k+1,\\
1/3 &\text{ if } j\ge k+2,
\end{cases}
\end{eqnarray*}
for $r\in\{0,1,2\}$. For this setting, we have 
\begin{align*}
3^{m-(k+1)}P(x) &=3^{m-(k+1)}\sum_{t_1\ldots, t_m\in \{0,1,2\}} \left(\prod_{\substack{j\in [m]:\\t_j=0}}p_j^0\prod_{\substack{j\in [m]:\\ t_j=1}}p_j^1\prod_{\substack{j\in [m]:\\t_j=2}}p_j^2\right)f_{t_1\ldots, t_m}^{(t)}(x+d)\\
&=\sum_{t_1\ldots, t_{k+1}\in \{0,1,2\}} \left(\prod_{\substack{j\in [k+1]:\\t_j=0}}p_j^0\prod_{\substack{j\in [k+1]:\\t_j=1}}p_j^1\prod_{\substack{j\in [k+1]:\\t_j=2}}p_j^2\right)f_{t_1\ldots, t_{k+1}}^{(t)}(x+d)\\
&=\sum_{t_1\ldots, t_{k}\in \{0,1,2\}} \left(\prod_{\substack{j\in [k]:\\t_j=0}}p_j^0\prod_{\substack{j\in [k]:\\t_j=1}}p_j^1\prod_{\substack{j\in [k]:\\t_j=2}}p_j^2\right) \cdot \sum_{r=0}^2 \alpha_r f_{t_1,\ldots,t_k,s_{k+1}=r}^{(t)}(x+d)\\
&=\alpha_0f_{s_1\ldots, s_{k},s_{k+1}=0}^{(t)}(x+d)+\alpha_1f_{s_1\ldots, s_{k},s_{k+1}=1}^{(t)}(x+d)+\alpha_2f_{s_1\ldots, s_{k},s_{k+1}=2}^{(t)}(x+d).
\end{align*}
By Theorem \ref{stronger-real-rootedness}, we know that $P(x)$ is real-rooted. Hence, $P(x-d)$ is also real-rooted and we have the conclusion.
\end{proof}

\begin{lemma}\label{real-rootedness-for-4}
Fix $t = i = \sqrt{-1}$. 
For every $k=0,1,\ldots, m-1$, every choice of $b_1,\ldots, b_m\in \{0,1\}$, every partial assignment $s_1,\ldots, s_k\in \{0,2\}$, and every non-negative $\alpha_0,\alpha_2$ such that $\alpha_0+\alpha_2=1$, the polynomial
\[
\alpha_0 f_{s_1,\ldots,s_k,s_{k+1}=0}^{(b,t)}(x) + \alpha_2 f_{s_1,\ldots,s_k,s_{k+1}=2}^{(b,t)}(x)
\]
has all real roots.
\end{lemma}
\begin{proof}
We will use Theorem \ref{stronger-real-rootedness} with appropriately chosen $a_j^r$ and $p_j^r$, for $j\in E$ and $r\in \H=\{0,2\}$. Fix $k\in \{0,1,\ldots,m-1\}$, partial assignment $s_1,\ldots, s_k\in\{0,2\}$, preliminary assignment $b_1,\ldots, b_m\in \{0,1\}$ and values $\alpha_0,\alpha_2\ge 0$ such that $\alpha_0+\alpha_2=1$. For edge $j=(u,v)$, we take $a_j^0=e_u-(-i)^{b(u,v)}e_v$ and $a_j^2=e_u+i^{b(u,v)}e_v$, where $e_u\in \reals^n$ is the indicator vector of vertex $u$. Moreover, we will take 
\begin{eqnarray*}
p_j^r=
\begin{cases}
1 &\text{ if } j\le k, s_j=r\\
0 &\text{ if } j\le k, s_j\in \{0,2\}\setminus \{r\} \\
\alpha_r &\text{ if } j= k+1,\\
1/2 &\text{ if } j\ge k+2,
\end{cases}
\end{eqnarray*}
for $r\in \{0,2\}$. For this setting, we have 
\begin{align*}
P(x)
&=\sum_{t_1\ldots, t_m\in \{0,2\}} \left(\prod_{j\in [m]:t_j=0}p_j^0\prod_{j\in [m]:t_j=2}p_j^2\right)f_{t_1\ldots, t_m}^{(b,t)}(x+d)\\
&=\frac{1}{2^{m-(k+1)}}\sum_{t_1,\ldots, t_{k+1}\in \{0,2\}} \left(\prod_{j\in [k+1]:t_j=0}p_j^0\prod_{j\in [k+1]:t_j=2}p_j^2\right)f_{t_1\ldots, t_{k+1}}^{(b,t)}(x+d)\\
&=\frac{1}{2^{m-(k+1)}}\sum_{t_1,\ldots, t_{k}\in \{0,2\}} \left(\prod_{j\in [k]:t_j=0}p_j^0\prod_{j\in [k]:t_j=2}p_j^2\right)\\
&\quad \quad \quad \quad \quad \quad \cdot \left(\alpha_0f_{t_1\ldots, t_{k},s_{k+1}=0}^{(b,t)} (x+d)+\alpha_2f_{t_1\ldots, t_{k},s_{k+1}=2}^{(b,t)}(x+d)\right)\\
&=\frac{1}{2^{m-(k+1)}}\left(\alpha_0f_{s_1\ldots, s_{k},s_{k+1}=0}^{(b,t)} (x+d)+\alpha_2f_{s_1\ldots, s_{k},s_{k+1}=2}^{(b,t)}(x+d)\right).
\end{align*}
Hence, 
\[
2^{m-(k+1)}P(x)=\alpha_0 f_{s_1,\ldots,s_k,s_{k+1}=0}(x+d)+\alpha_2 f_{s_1,\ldots,s_k,s_{k+1}=2}(x+d).
\]
By Theorem \ref{stronger-real-rootedness}, we know that $P(x)$ is real-rooted. Hence, $P(x-d)$ is also real-rooted and we have the conclusion.
\end{proof}

The following two theorems complete the proof of Theorem \ref{3-4-lifts-preserve-Ramanujan}.

\begin{theorem}\label{3-lifts-preserve-Ramanujan}
If $G=(V,E)$ is a $d$-regular bipartite Ramanujan graph, then there exists a shift $3$-lift of $G$ that is also $d$-regular, bipartite, and Ramanujan.
\end{theorem}
\begin{proof}
Let $\mathcal{H} = \{0,1,2\}$ and $t = \omega = e^{2\pi i/3}$.
By Lemma \ref{matching-polynomial-3-lifts}, we have $f_{\emptyset}^{(t)}(x)= 3^m\mu_G(x)$. By Theorem \ref{matching-polynomial-roots}, the largest root of $\mu_G(x)$ is at most $2\sqrt{d-1}$. 
Let us consider the family of polynomials $f_{s_1,\ldots, s_k}^{(t)}(x)$. For every $k=0,1,\ldots, m-1$ and every $s_1,\ldots, s_k\in \{0,1,2\}$, the polynomials $f_{s_1,\ldots, s_k}^{(t)}(x)$ have positive leading-coefficients by definition. Lemmas \ref{conv-comb-implies-interlacing} and \ref{real-rootedness-for-3} imply that the polynomials
\[
\left(f_{s_1,\ldots,s_k,s_{k+1}=r}^{(t)}(x)\right)_{r\in \Ht}
\]
are also real-rooted and have a common interlacing. Thus, by Theorem \ref{common-interlacing-implies-good-signing}, there exists a shift function $s$ with $s_j\in \{0,1,2\}\,\forall j\in [m]$ such that the largest root of $\det[xI-A_s(\omega)]$ is at most $2\sqrt{d-1}$. 

Moreover $A_s(\omega)=A_s(\omega^2)^T$. Since the determinant of a matrix is preserved under the transpose operation, the characteristic polynomials of $A_s(\omega)$ and $A_s(\omega^2)$ are identical. Thus, for the shift $s$, the largest of the eigenvalues of $A_s(\omega)$ and $A_s(\omega^2)$ is at most $2\sqrt{d-1}$. Hence, by Theorem \ref{eigenvalues-of-shifts}, we have a shift 3-lift that preserves the Ramanujan property. It is straightforward to verify that any lift preserves the $d$-regular and  bipartiteness properties. 
\end{proof}

\begin{theorem}\label{4-lifts-preserve-Ramanujan}
If $G=(V,E)$ is a $d$-regular bipartite Ramanujan graph, then there exists a shift $4$-lift of $G$ that is also $d$-regular, bipartite, and Ramanujan.
\end{theorem}
\begin{proof}
Let $\mathcal{H} = \{0,2\}$ and $t = i = \sqrt{-1}$.
By Theorem \ref{thm:2-lift}, we have a shift $2$-lift with shift function $b:E\to\{0,1\}$ such that the maximum eigenvalue of $A_b(-1)$ is at most $2\sqrt{d-1}$. Let us fix this function $b$ and consider the family of polynomials $\left(f_{s_1,\ldots, s_k}^{(b,t)}(x)\right)_{s_1,\ldots,s_k\in \mathcal{H}}$. By Lemma \ref{matching-polynomial-4-lifts}, we have that $f_{\emptyset}^{(b,t)}(x)= 2^m \mu_G(x)$. By Theorem \ref{matching-polynomial-roots}, all roots of $\mu_G(x)$ are at most $2\sqrt{d-1}$. For every $k=0,1,\ldots, m-1$ and every $s_1,\ldots, s_k\in \mathcal{H}$, the polynomials $f_{s_1,\ldots, s_k}^{(b,t)}(x)$ have positive leading coefficients by definition. By Lemmas \ref{conv-comb-implies-interlacing} and \ref{real-rootedness-for-4}, the polynomials 
\[
\left(f_{s_1,\ldots, s_k, s_{k+1}=r}^{(b,t)}(x)\right)_{r\in \mathcal{H}} 
\]
are real-rooted and have a common interlacing. Therefore, by Theorem \ref{common-interlacing-implies-good-signing}, we have a shift function $s$ with $s_j\in \{0,2\}\,\forall j\in [m]$ such that the largest root of $\det[xI-A_{b+s}(i)]$ is at most $2\sqrt{d-1}$. Thus, by considering $s'=b+s$, we have a shift function $s': E\to \{0,1,2,3\}$ such that the largest root of $\det[xI-A_{s'}(i)]$ is at most $2\sqrt{d-1}$. 

Furthermore, $A_{s'}(-i)=A_{s'}(i)^T$ and $A_{s'}(-1)=A_b(-1)$ since $(-1)^{s'(e)} = (-1)^{s(e) + b(e)} = (-1)^{b(e)}$ for all $e\in E$. Therefore, the eigenvalues of $A_{s'}(-1)$, $A_{s'}(i)$ and $A_{s'}(-i)$ are at most $2\sqrt{d-1}$. Hence, by Theorem \ref{eigenvalues-of-shifts}, we have a shift 4-lift given by $s'$ that preserves the Ramanujan property. 
\end{proof}

\section{Discussion} \label{sec:discussion}
In this work, we have considered an alternative approach to constructing Ramanujan graphs efficiently. Instead of repeatedly taking $2$-lifts, we have suggested taking a shift $k$-lift for $k$ growing exponentially in $n$, where $n$ is the number of vertices in the base graph. The existence of a shift $k$-lift that preserves the Ramanujan property for such values of $k$ would immediately lead to a faster algorithm than repeatedly taking $2$-lifts. We take a first step towards proving the existence of such shift $k$-lifts by showing it for $k=3$ and $k=4$. 

A more general approach would be to consider other subfamilies of permutations over $k$ elements as opposed to considering shift permutations. Subfamilies that arise from subgroups of the permutation group $S_k$ have a convenient characterization of the new eigenvalues of the lifted graph. In particular, the family of shift permutations correspond to the subgroups $\Z/k\Z$ of $S_k$~\cite{MS95}. In order to obtain faster construction of Ramanujan graphs from alternative subgroups $\Gamma$, we need the following requirement on the subgroup $\Gamma$: (1) the size of the subgroup $\Gamma$ grows polynomial in $k$, and (2) there exists a $k$-lift in the subgroup preserving the Ramanujan property of $n$-vertex Ramanujan graphs for $k$ being superpolynomial in $n$. Natural candidates to consider are abelian subgroups of order $k$.\\



\noindent {\bf Acknowledgments.}
The authors would like to thank Salil Vadhan and Jelani Nelson for many helpful conversations related to the topic.

\bibliographystyle{abbrv}
\bibliography{references}

\begin{thebibliography}{10}

\bibitem{AKM}
N.~Agarwal, A.~Kolla, and V.~Madan.
\newblock Small lifts of expander graphs are expanding.
\newblock {\em CoRR}, abs/1311.3268, 2013.

\bibitem{AKS83}
M.~Ajtai, J.~Koml\'{o}s, and E.~Szemer\'{e}di.
\newblock An {O}$(n \log n)$ sorting network.
\newblock In {\em Proceedings of the fifteenth annual ACM symposium on Theory
  of computing}, STOC '83, pages 1--9, 1983.

\bibitem{ABN92}
N.~Alon, J.~Brucka, J.~Naor, M.~Naor, and R.~M. Roth.
\newblock Construction of asymptotically good low-rate error-correcting codes
  through pseudo-random graphs.
\newblock {\em IEEE Transactions on Information Theory}, 38(2):509--516, 1992.

\bibitem{BL06}
Y.~Bilu and N.~Linial.
\newblock Lifts, discrepancy and nearly optimal spectral gap.
\newblock {\em Combinatorica}, 26(5):495--519, 2006.

\bibitem{crypto}
D.~X. Charles, E.~Z. Goren, and K.~Lauter.
\newblock Cryptographic hash functions from expander graphs.
\newblock {\em IACR Cryptology ePrint Archive}, 2006:21, 2006.

\bibitem{Chi92}
P.~Chiu.
\newblock Cubic ramanujan graphs.
\newblock {\em Combinatorica}, 12(3):275--285, 1992.

\bibitem{Din07}
I.~Dinur.
\newblock The {PCP} theorem by gap amplification.
\newblock {\em J. ACM}, 54(3):12, 2007.

\bibitem{Fisk}
S.~Fisk.
\newblock {\em {Polynomials, roots, and interlacing}}.
\newblock arXiv:math/0612833 [math.CA], 2008.

\bibitem{Guru09}
V.~Guruswami, C.~Umans, and S.~P. Vadhan.
\newblock Unbalanced expanders and randomness extractors from
  {P}arvaresh--{V}ardy codes.
\newblock {\em J. ACM}, 56(4), 2009.

\bibitem{HL72}
O.~Heilmann and E.~Lieb.
\newblock {Theory of monomer-dimer systems}.
\newblock {\em Communications in Mathematical Physics}, 25(3):190--232, 1972.

\bibitem{Has99}
J.~H\r{a}stad, R.~Impagliazzo, L.~A. Levin, and M.~Luby.
\newblock A pseudorandom generator from any one-way function.
\newblock {\em SIAM Journal on Computing}, 28:12--24, 1999.

\bibitem{JL97}
B.~W. Jordan and R.~Livne.
\newblock Ramanujan local systems on graphs.
\newblock {\em Topology}, 36(5):1007 -- 1024, 1997.

\bibitem{LPV14}
S.~Liu, N.~Peyerimhoff, and A.~Vdovina.
\newblock {Signatures, lifts, and eigenvalues of graphs}.
\newblock {\em CoRR}, abs/1412.6841, 2014.

\bibitem{LPS88}
A.~Lubotzky, R.~Phillips, and P.~Sarnak.
\newblock Ramanujan graphs.
\newblock {\em Combinatorica}, 8(3):261--277, 1988.

\bibitem{MSS13}
A.~Marcus, D.~A. Spielman, and N.~Srivastava.
\newblock {Interlacing Families I: Bipartite Ramanujan Graphs of All Degrees}.
\newblock In {\em FOCS}, pages 529--537, 2013.

\bibitem{MSS13-2}
A.~Marcus, D.~A. Spielman, and N.~Srivastava.
\newblock {Interlacing Families II: Mixed Characteristic Polynomials and the
  Kadison-Singer Problem}.
\newblock {\em CoRR}, abs/1306.3969, 2013.

\bibitem{Mar88}
G.~Margulis.
\newblock Explicit group-theoretic constructions of combinatorial schemes and
  their applications in the construction of expanders and concentrators.
\newblock {\em Probl. Inf. Transm}, 24(1):39--46, 1988.

\bibitem{MS95}
H.~Mizuno and I.~Sato.
\newblock Characteristic polynomials of some graph coverings.
\newblock {\em Discrete Mathematics}, 142(13):295--298, 1995.

\bibitem{Mor94}
M.~Morgenstern.
\newblock Existence and explicit constructions of q + 1 regular ramanujan
  graphs for every prime power q.
\newblock {\em Journal of Combinatorial Theory, Series B}, 62(1):44 -- 62,
  1994.

\bibitem{Nil91}
A.~Nilli.
\newblock On the second eigenvalue of a graph.
\newblock {\em Discrete Math}, 91(2):207--210, 1991.

\bibitem{Piz90}
A.~Pizer.
\newblock Ramanujan graphs and hecke operators.
\newblock {\em Bull. Amer. Math. Soc.}, 23:127--137, 1990.

\bibitem{Spi96}
M.~Sipser and D.~A. Spielman.
\newblock Expander codes.
\newblock {\em IEEE Transactions on Information Theory}, 42:1710--1722, 1996.

\bibitem{Zem94}
G.~Zemor.
\newblock Hash functions and {C}ayley graphs.
\newblock {\em Designs, Codes and Cryptography}, 4:381--394, 1994.

\end{thebibliography}

\end{document}